\documentclass[12pt]{amsart}

\usepackage{color}
\usepackage{amsmath}
\usepackage{amsthm}
\usepackage{amssymb}
\usepackage{graphicx}
\usepackage{enumerate}
\usepackage{amsfonts}
\usepackage{mathrsfs}
\usepackage{parskip}
\usepackage{mathdots}
\usepackage{color}

\numberwithin{equation}{section}
\theoremstyle{plain}
\newtheorem{Proposition}[equation]{Proposition}

\newtheorem*{Corollary*}{Corollary}
\newtheorem{Theorem}[equation]{Theorem}
\newtheorem*{Theorem*}{Theorem}
\newtheorem{Lemma}[equation]{Lemma}
\theoremstyle{definition}

\newtheorem{Remark}[equation]{Remark}

\def\D{\mathbb{D}}
\def\T{\mathbb{T}}

\def\K{\mathcal{K}}
\def\phi{\varphi}

\title{Compact operators on model spaces}

\author[Chalendar]{Isabelle Chalendar}
	\address{Department of Mathematics, Institute Camille Jordan, Universit\'{e} Claude-Bernard (Lyon I), 69622 Villeurbanne, France}
	\email{chalendar@math.univ-lyon1.fr}

\author[Ross]{William T. Ross}
	\address{Department of Mathematics and Computer Science, University of Richmond, Richmond, VA 23173, USA}
	\email{wross@richmond.edu}

\keywords{Hardy spaces, inner functions, model spaces, compact operators, Toeplitz operators}

\subjclass[2010]{30J05, 30H10, 46E22}

\begin{document}

\begin{abstract}
We give a characterization of the compact operators on a model space in terms of asymptotic Toeplitz operators. 
\end{abstract}

\maketitle

\section{Introduction}

If $H^2$ denotes the classical Hardy space of the open unit disk $\D$ \cite{Duren, Garnett}, a theorem of Brown and Halmos \cite{BH} says that a bounded linear operator $T$ on $H^2$ is a Toeplitz operator if and only if 
$$S^{*} T S = T,$$
where $S f = z f$ is the well-known unilateral shift on $H^2$. By a {\em Toeplitz operator} \cite{Bottcher}, we mean, for a given symbol $\varphi \in L^{\infty}(\T, m)$ ($\T$ is the unit circle and $m$ is normalized Lebesgue measure on $\T$), the operator 
$$T_{\phi}: H^2 \to H^2, \quad T_{\phi} f = P(\phi f),$$ where $P$ is the orthogonal projection of $L^2$ onto $H^2$.

This notion of ``Toeplitzness'' was extended in various ways. Barria and Halmos \cite{MR667164} examined the so-called {\em asymptotically Toeplitz operators} operators $T$ on $H^2$ for which the sequence of operators 
$$\{S^{* n} T S^{n}\}_{n \geqslant 1}$$ converges strongly. This class certainly includes the Toeplitz operators but also includes other operators such as those in the Hankel algebra. Feintuch \cite{MR1038338} discovered that one need not restrict to strong convergence of $\{S^{* n} T S^{n}\}_{n \geqslant 1}$ and worthwhile classes of operators arise from  the weak and uniform (or norm) limits of this sequence. Indeed, an operator $T$ on $H^2$ is uniformly asymptotically Toeplitz, i.e., $S^{*  n} T S^{n}$ converges in operator norm, if and only if 
\begin{equation}\label{sndjfkeruy}
T = T_1 + K,
\end{equation}
where $T_1$ is a Toeplitz operator, i.e., $S^{*} T_{1} S = T_1$, and $K$ is a compact operator on $H^2$.  Nazarov and Shapiro \cite{MR2297770} examined other associated notions of ``Toeplitzness'' with regards to certain composition operators on $H^2$. 

In this paper we explore a model space setting for this ``Toeplitzness'' discussion. For an inner function $\Theta$ on $\D$ (i.e., a bounded analytic function on $\D$ whose radial boundary values are unimodular almost everywhere on $\T$), one can define the {\em model space}  \cite{MSGMR, Niktr}
$$\K_{\Theta} = H^2 \ominus \Theta H^2.$$ Beurling's theorem \cite{Duren} says that these spaces are the generic invariant subspaces for the backward shift operator 
$$S^{*} f = \frac{f - f(0)}{z}$$ on $H^2$. By model theory for contractions \cite{Niktr}, certain types of Hilbert space contractions are unitarily equivalent to compressed shifts 
$$S_{\Theta} = P_{\Theta} S|_{\K_{\Theta}},$$
where $P_{\Theta}$ is the orthogonal projection of $L^2$ onto $\K_{\Theta}$.

In this model spaces setting, we examine, for a bounded operator $A$ on $\K_{\Theta}$, the sequence 
$$\{S_{\Theta}^{* n} A S_{\Theta}^{n}\}_{n \geqslant 1}.$$ Here we have a similar result as before (see Lemma \ref{77235bbx} below) in that $S_{\Theta}^{* n} A S_{\Theta}^{n}$ converges in operator norm if and only if  
$$A = A_{1} + K,$$
where $K$ is a compact operator on $\K_{\Theta}$ and $A_1$ satisfies $S_{\Theta}^{*} A_1 S_{\Theta} = A_1$. In the analogous $H^2$ setting, the operator $T_1$ from \eqref{sndjfkeruy} is a Toeplitz operator. In the model space setting, the corresponding operator $A_1$ is severely restricted. Indeed,
$$A_1 \equiv 0.$$ Thus, as the main theorem of this paper, we have the following characterization of the compact operators on $\K_{\Theta}$. 

\begin{Theorem}\label{MT}
For an inner function $\Theta$ and a bounded linear operator $A$ on $\K_{\Theta}$, the following are equivalent: 
\begin{enumerate}
\item[(i)] The sequence $S_{\Theta}^{* n} A S_{\Theta}^{n}$ converges in operator norm; 
\item[(ii)] $S_{\Theta}^{*n} A S_{\Theta}^{n} \to 0$ in operator norm;
\item[(iii)] $A$ is a compact operator. 
\end{enumerate}
\end{Theorem}

One can also explore the convergence of the sequence $S_{\Theta}^{* n} A S_{\Theta}^{n}$ in other topologies, such as the strong/weak operator topologies. Surprisingly what happens is entirely different from what happens in $H^2$.  

\begin{Proposition}\label{secondMT}
	For any inner function $\Theta$ and any bounded linear operator $A$ on $\K_{\Theta}$, the sequence $S_{\Theta}^{*n} A S_{\Theta}^{n}$ converges to zero strongly.
\end{Proposition}
In other words, the convergence of $S_{\Theta}^{* n} A S_{\Theta}^{n}$ in the strong or weak topology is always true (and to the same operator) and provides no information about $A$. 

%

\section{Characterization of the compact operators}

The following lemma proves the implication $(iii) \implies (ii)$ of Theorem \ref{MT}. 

\begin{Lemma}\label{43092367}
If $K$ is a compact operator on $\K_{\Theta}$ then 
$$\lim_{n \to \infty} \|S_{\Theta}^{* n} K S_{\Theta}^{n}\| = 0.$$
\end{Lemma}

\begin{proof}
Let $\mathcal{B}_{\Theta} = \{f \in \K_{\Theta}: \|f\| \leqslant 1\}$ denote the closed unit ball in $\K_{\Theta}$. First observe that 
$$\|S_{\Theta}^{n}\| \leqslant \|S_{\Theta}\|^{n} \leqslant \|P_{\Theta} S|_{\K_{\Theta}}\|^n \leqslant \|S\|^n = 1.$$ From here we see that 
\begin{align}
\|S_{\Theta}^{* n} K S_{\Theta}^{n}\| & = \sup_{f \in \mathcal{B}_{\Theta}} \|S_{\Theta}^{* n} K S_{\Theta}^{n} f\| \nonumber\\
& \leqslant \sup_{g \in \mathcal{B}_{\Theta}} \|S_{\Theta}^{* n} K g\| \nonumber\\
& \leqslant \sup_{h \in \overline{K(\mathcal{B}_{\Theta}})} \|S_{\Theta}^{* n} h\|. \label{73652rituy4rejkg}
\end{align}

Second, note that $S^{* n} \to 0$ strongly. Indeed, if $f = \sum_{k \geqslant 0} a_k z^k \in H^2$, then 
$$
\|S^{* n} f\|^2 = \sum_{k \geqslant n + 1} |a_k|^2 \to 0 \quad n \to \infty.
$$ Thus since $S_{\Theta}^{* n} = S^{* n}|_{\K_{\Theta}}$ (since $\K_{\Theta}$ is $S^{*}$-invariant), we see that 
\begin{equation}\label{sdfsdf4r4rewf999}
S_{\Theta}^{* n} \to 0 \; \; \mbox{strongly}.
\end{equation}

Let $\epsilon > 0$ be given and let $h \in \overline{K(\mathcal{B}_{\Theta})}$. Since $S_{\Theta}^{* n} \to 0$ strongly, there exists an $n_{h, \epsilon}$ such that $\|S_{\Theta}^{* n} h\| < \epsilon/2$ for all $n > n_{h, \epsilon}$. The continuity of the operator $S^{* n_{h, \epsilon}}$ implies that there exists a $r_{h, \epsilon}$ such that for all $q$ belonging to 
$$B(h, r_{h, \epsilon}) = \{q \in \K_{\Theta}: \|q - h\| < r_{h, \epsilon}\}$$ we have $\|S^{* n_{h, \epsilon}} q\| < \epsilon$. 

Again using the fact that $\|S_{\Theta}^{*}\| \leqslant 1$, we see that for all $q \in B(h, r_{h, \epsilon})$ and all $n > n_{h, \epsilon}$ we have 
\begin{equation}\label{oweijfhd}
\|S_{\Theta}^{* n} q\| = \|S_{\Theta}^{* (n - n_{h, \epsilon}) }S_{\Theta}^{* n_{h, \epsilon}} q\| \leqslant \|S_{\Theta}^{* n_{h, \epsilon}} q\| < \epsilon.
\end{equation}
Moreover, we have 
$$
\overline{K(\mathcal{B}_{\Theta})} \subset \bigcup_{h \in \overline{K(\mathcal{B}_{\Theta})}} B(h, r_{h, \epsilon}).
$$
The compactness of $\overline{K(\mathcal{B}_{\Theta})}$ implies that there exists $h_{1}, \ldots, h_{N}$ $(N = N_{\epsilon}$) belonging to $\overline{K(\mathcal{B}_{\Theta})}$ such that 
$$\overline{K(\mathcal{B}_{\Theta})} \subset \bigcup_{k = 1}^{N} B(h_k, r_{h_k, \epsilon}).$$
For all $n > \max\{n_{h_1, \epsilon}, \ldots, n_{h_{N}, \epsilon}\}$ we use \eqref{oweijfhd} along with \eqref{73652rituy4rejkg} to see that 
$$\|S_{\Theta}^{* n} h\| < \epsilon \quad \forall h \in \overline{K(\mathcal{B}_{\Theta})}.$$
This proves the lemma. 
\end{proof}

\begin{Remark}\label{99e776ewr}
Important to the proof above was the fact that $S_{\Theta}^{* n} \to 0$ strongly (see \eqref{sdfsdf4r4rewf999}). One can show that $S_{\Theta}$ is unitarily equivalent to $S_{\Psi}^{*}$, where $\Psi$ is the inner function defined by $\Psi(z) = \overline{\Theta(\overline{z})}$ \cite[p.~303]{MSGMR}. From here we see that $S_{\Theta}^{n} \to 0$ strongly. This detail will be important at the end of the paper in the proof of Theorem~\ref{secondMT}. 
\end{Remark}

\begin{Lemma}\label{77235bbx}
Suppose $T$ is a bounded operator on $\K_{\Theta}$ such that 
$S_{\Theta}^{* n} T S_{\Theta}^{n}$ converges in norm. 
Then $T = T_{1} + K$, where $K$ is a compact operator on $\K_{\Theta}$ and $T_1$ is a bounded operator on $\K_{\Theta}$ satisfying $S_{\Theta}^{*} T_{1} S_{\Theta} = T_{1}$.
\end{Lemma}

\begin{proof}
Let $A$ be a bounded operator on $\K_{\Theta}$ such that  $$\|S_{\Theta}^{* n} T S_{\Theta}^{n} - A\| \to 0.$$ Then 
\begin{align*}
\|S_{\Theta}^{* (n + 1)} T S_{\Theta}^{n + 1} - S_{\Theta}^{*} A S_{\Theta}\| & = \|S_{\Theta}^{*}(S_{\Theta}^{* n} T S_{\Theta}^{n} - A) S_{\Theta}\|\\
& \leqslant \|S_{\Theta}^{* n} T S_{\Theta}^{n} - A\| \to 0.
\end{align*}
This implies that 
\begin{equation}\label{32897kjdfs}
S_{\Theta}^{*} A S_{\Theta} = A.
\end{equation} From here it follows that 
\begin{equation}\label{sdjkiufer84}
S_{\Theta}^{* n} T S_{\Theta}^{n} - A = S_{\Theta}^{* n} (T - A) S_{\Theta}^{n}, \quad  n \geqslant 0.
\end{equation}

Define 
$$P_{n} := S_{\Theta}^{n} S_{\Theta}^{* n} \; \; \mbox{and} \; \;  Q_{n} := I - P_{n} = I - S_{\Theta}^{n} S_{\Theta}^{* n}$$ and observe that 
\begin{align}
P_{n} (T - A) P_{n} & = (T - A) \nonumber\\
& -Q_{n} (T - A) + Q_{n} (T - A) Q_{n} - (T - A) Q_n. \label{34q9857634895}
\end{align}
Furthermore by \eqref{sdjkiufer84} we have 
\begin{align*}
\|P_{n} (T - A) P_{n}\| & = \|S_{\Theta}^{n} S_{\Theta}^{*n} (T - A) S_{\Theta}^{n} S_{\Theta}^{* n}\|\\
& \leqslant \|S_{\Theta}^{* n} (T - A) S_{\Theta}^{n}\|\\
& = \|S_{\Theta}^{* n} T S_{\Theta}^{n} - A\| \to 0.
\end{align*}

If
$$k_{\lambda}(z) = \frac{1 - \overline{\Theta(\lambda)} \Theta(z)}{1 - \overline{\lambda} z}, \quad \lambda, z \in \D,$$
is the reproducing kernel for $\K_{\Theta}$, then \cite[p.~497]{Sarason} gives us the well-known operator identity 
$$S_{\Theta}S_{\Theta}^{*} = I - k_{0} \otimes k_{0}.$$ Iterating the above $n$ times we get  
$$S_{\Theta}^{n} S_{\Theta}^{*n} = I - \sum_{j = 0}^{n - 1} S_{\Theta}^{j} k_{0} \otimes S_{\Theta}^{* j} k_0.$$
In other words, 
$$Q_n = I - P_{n} = \sum_{j = 0}^{n - 1} S_{\Theta}^{j} k_{0} \otimes S_{\Theta}^{* j} k_0$$ is a finite rank operator.

By \eqref{34q9857634895} this means that 
$$F_{n} := -Q_{n} (T - A) + Q_{n} (T - A) Q_{n} - (T - A) Q_n$$ is a finite rank operator which converges in norm to $A - T$. Hence $A - T$ a compact operator and, by \eqref{32897kjdfs}, $A$ satisfies $S_{\Theta}^{*} A S_{\Theta} = A$.
\end{proof}

So far we know from Lemma \ref{43092367} that every compact operator $K$ on $\K_{\Theta}$ satisfies 
$$\lim_{n \to \infty} \|S_{\Theta}^{* n} K S_{\Theta}^{n}\| = 0.$$
Furthermore, from Lemma \ref{77235bbx} we see that an operator $A$ for which $S_{\Theta}^{* n} A S_{\Theta}^{n}$ converges in operator norm can be written as 
$A = A_{1} + K$ where $K$ is compact and $A_1$ satisfies $S_{\Theta}^{*} A_1 S_{\Theta} = A_1$. To complete the proof of Theorem \ref{MT}, we need to show that $$S_{\Theta}^{*} A S_{\Theta} = A \iff A \equiv 0.$$ This is done with the following result. 

\begin{Proposition}\label{lem:zero}
Suppose $A$ is a bounded operator on $\K_{\Theta}$. Then $S_{\Theta}^{*} A S_{\Theta} = A$ if and only if $A \equiv 0$. 
\end{Proposition}

\begin{proof}
Recall that 
$$k_{\lambda}(z) = \frac{1 - \overline{\Theta(\lambda)} \Theta(z)}{1 - \overline{\lambda} z}$$ is the kernel function for $\K_{\Theta}$. There is also the ``conjugate kernel'' 
$$\widetilde{k}_{\lambda}(z) = \frac{\Theta(z) - \Theta(\lambda)}{z - \lambda}$$ which also belongs to $\K_{\Theta}$ \cite[p.~495]{Sarason}. 
The proof depends on the following kernel function identities from \cite[p.~496]{Sarason}: 
$$S_{\Theta} \widetilde{k}_{\lambda} = \lambda \widetilde{k}_{\lambda} - \Theta(\lambda) k_{0},$$
$$S_{\Theta} k_{\lambda} = \frac{1}{\overline{\lambda}} k_{\lambda} - \frac{1}{\overline{\lambda}} k_0.$$
This gives us 
\begin{align*}
(A \widetilde{k}_{\lambda})(z) & = \langle S_{\Theta}^{*} A S_{\Theta} \widetilde{k}_{\lambda}, k_{z}\rangle\\
& = \langle A S_{\Theta} \widetilde{k}_{\lambda}, S_{\Theta} k_z\rangle\\
& = \langle A (\lambda \widetilde{k}_{\lambda} - \Theta(\lambda) k_{0}), \frac{1}{\overline{z}} k_{z} - \frac{1}{\overline{z}} k_0\rangle\\
& = \frac{\lambda}{z} (A \widetilde{k}_{\lambda})(z) - \frac{\Theta(\lambda)}{z} (A k_0)(z) - \frac{\lambda}{z} (A \widetilde{k}_{\lambda})(0) + \frac{\Theta(\lambda)}{z} (A k_{0})(0).
\end{align*}
Re-arrange the above identity:
$$(A \widetilde{k}_{\lambda})(z) (1 - \frac{\lambda}{z}) = -\frac{\Theta(\lambda)}{z} (A k_0)(z) - \frac{\lambda}{z} (A \widetilde{k}_{\lambda})(0) + \frac{\Theta(\lambda)}{z} (A k_0)(0).$$
Multiply through by $z$:
$$(z - \lambda) (A \widetilde{k}_{\lambda})(z) = -\Theta(\lambda) (A k_0)(z) - \lambda (A \widetilde{k}_{\lambda})(0) + \Theta(\lambda) (A k_0)(0).$$
Divide by $(z - \lambda)$ and re-arrange: 
\begin{equation}\label{2049tu}
(A \widetilde{k}_{\lambda})(z) = - \Theta(\lambda) \left( \frac{(A k_0)(z) - (A k_0)(\lambda)}{z - \lambda}\right) - \lambda \frac{(A \widetilde{k}_{\lambda})(0)}{z - \lambda}.
\end{equation}
Observe that the functions 
$$(A \widetilde{k}_{\lambda})(z) \; \; \mbox{and} \; \; \frac{(A k_0)(z) - (A k_0)(\lambda)}{z - \lambda}$$
belong to $\K_{\Theta}$ for all $\lambda \in \D$. This means that 
$$\lambda \frac{(A \widetilde{k}_{\lambda})(0)}{z - \lambda}$$
must  also belong to $\K_{\Theta}$ for all $\lambda \in \D$ which means (since there is an obvious pole at $z = \lambda$) that 
\begin{equation}\label{csdnjf0}
(A  \widetilde{k}_{\lambda})(0) = 0.
\end{equation}

The identity in \eqref{2049tu} can now be written as 
\begin{equation}\label{0-9dsn66}
(A \widetilde{k}_{\lambda})(z) = - \Theta(\lambda) \left( \frac{(A k_0)(z) - (A k_0)(\lambda)}{z - \lambda}\right).
\end{equation}
Plug in $z = 0$ into the previous identity and use \eqref{csdnjf0} to see that 
$$0 = (A \widetilde{k}_{\lambda})(0) = \frac{\Theta(\lambda)}{\lambda} ((A k_0)(0) - (A k_0)(\lambda)), \quad \lambda \in \D.$$
Since $\Theta$ is not the zero function, we get 
\begin{equation}\label{0o-asfhg0-}
(A k_0)(\lambda) = (A k_0)(0), \quad  \lambda \in \D.
\end{equation}
Plus this into \eqref{0-9dsn66} to get that 
$$A \widetilde{k}_{\lambda} = 0 \quad \forall \lambda \in \D.$$
But since the linear span of these conjugate kernels form a dense subset in $\K_{\Theta}$ (the conjugation operator $f \mapsto \widetilde{f}$ is isometric and involutive \cite[p.~495]{Sarason}), we see that $A$ must be the zero operator. 
\end{proof}

\begin{proof}[Proof of Proposition~\ref{secondMT}]
	
	For any $f, g, \in \K_{\Theta}$  and $n \geqslant 0$ we have 
	\begin{equation}\label{999sss}
	|\langle S^{* n} AS^{n} f, g\rangle| = |\langle S^{n}_{\Theta} f, A^*S^{n}_{\Theta} g\rangle| \leqslant\|S^{n}_{\Theta} f\| \|A^* S^{n}_{\Theta} g\|.
	\end{equation}
	Taking the supremum in (\ref{999sss}) over $g\in \K_\Theta$ with $\|g\|\leqslant 1$, and using the fact that $\|S_{\Theta}\| \leqslant 1$,  we get 
	\begin{equation}\label{999}
	\| S^{* n} A S^{n} f\| \leqslant  \| S^{n}_{\Theta} f\| \|A^{*}\|.
	\end{equation}
	  From Remark \ref{99e776ewr}, we conclude that the right hand side of \eqref{999sss} goes to zero as $n \to \infty$. Thus $S^{*n}_{\Theta} A S^{n}_{\Theta} \to 0$ strongly.	
\end{proof}



\bibliographystyle{plain}
\def\cprime{$'$}


\end{document}